\documentclass[11pt, a4paper]{amsart}
\usepackage{amssymb, mathtools, enumerate}
\usepackage[colorlinks=true,allcolors=magenta]{hyperref}
\mathtoolsset{showonlyrefs=true}

\DeclareMathOperator*{\R}{Re}

\newcommand{\dd}{\mathrm{d}}

\newcommand{\B}{\mathcal{B}}
\newcommand{\T}{(T(t))_{t\ge0}}
\newcommand{\TB}{(T_B(t))_{t\ge0}}

\newcommand{\RR}{\mathbb{R}}
\newcommand{\CC}{\mathbb{C}}
\newcommand{\ZZ}{\mathbb{Z}}

\newcommand{\NN}{\mathbb{N}}

\newcommand{\inv}{^{-1}}

\newtheorem{thm}{Theorem}[section]
\newtheorem{prp}[thm]{Proposition}
\newtheorem{lem}[thm]{Lemma}
\newtheorem{cor}[thm]{Corollary}
\theoremstyle{definition}
\newtheorem{rem}[thm]{Remark}

\numberwithin{equation}{section} 

\begin{document}

\title[A non-uniform Datko--Pazy theorem]{A non-uniform Datko--Pazy theorem for bounded operator semigroups}

\author[L.~Paunonen]{Lassi Paunonen}
\address[L.~Paunonen]{Mathematics, Faculty of Information Technology and Communication Sciences, Tampere University, PO~ Box 692, 33101 Tampere, Finland}
 \email{lassi.paunonen@tuni.fi}

\author[D.~Seifert]{David Seifert}
\address[D.~Seifert]{School of Mathematics, Statistics and Physics, Newcastle University, Herschel Building, Newcastle upon Tyne, NE1 7RU, United Kingdom}
\email{david.seifert@ncl.ac.uk}

\author[N.~Vanspranghe]{Nicolas Vanspranghe}
\address[N.~Vanspranghe]{Mathematics, Faculty of Information Technology and Communication Sciences, Tampere University, PO~ Box 692, 33101 Tampere, Finland}
 \email{nicolas.vanspranghe@tuni.fi}

\begin{abstract}
We present a non-uniform analogue of the classical Datko--Pazy theorem. 
Our main result shows that an integrability condition imposed on orbits originating in a fractional domain of the generator   (as opposed to all orbits)  implies polynomial stability of  a bounded $C_0$-semigroup.
 As an application of this result we establish polynomial stability of a semigroup under a certain non-uniform Lyapunov-type condition. We moreover give
a new proof, under slightly weaker assumptions, of a recent result deducing polynomial stability from a certain non-uniform observability condition.
\end{abstract}


\maketitle

\section{Introduction}\label{sec:int}

Recall that, given  a $C_0$-semigroup $\T$  of bounded linear operators on some Banach space, $\T$ is said to be \emph{uniformly exponentially stable} if there exist constants $M,\omega>0$ such that $\|T(t)\|\le Me^{-\omega t}$ for all $t\ge0$. The famous Datko--Pazy theorem provides a necessary and sufficient condition for a $C_0$-semigroup to be uniformly exponentially stable. We state the result as follows, referring the reader to~\cite[Thm.~5.1.2]{ABHN11} for both a proof of the Datko--Pazy theorem and for several further equivalent conditions. 

\begin{thm} \label{thm:Datko}
Let $\T$ be a $C_0$-semigroup on a Banach space $X$, and let $1\le p<\infty$. The following are equivalent:
\begin{enumerate}[(i)]
\item[\textup{(i)}]  $\T$ is uniformly exponentially stable;
\item[\textup{(ii)}] $T(\cdot)x\in L^p(0,\infty;X)$ for all $x\in X$.
\end{enumerate}
\end{thm}

In many applications of semigroup theory, for instance to the study of energy decay of damped waves, uniform exponential stability is often too much to hope for and one is instead interested in weaker notions of stability. One important notion in this context is \emph{non-uniform stability}, sometimes called \emph{semi-uniform stability}. Let $\T$ be a $C_0$-semigroup with generator $A$. Recall that the semigroup is said to be \emph{bounded} if $\sup_{t\ge0}\|T(t)\|<\infty$, and a bounded $C_0$-semigroup $\T$ is said to be \emph{non-uniformly stable} if $\|T(t)(\lambda-A)\inv\|\to0$ as $t\to\infty$ for some, or equivalently all, $\lambda\in\rho(A)$. In the case where $\|T(t)(\lambda-A)\inv\|=O(t^{-\alpha})$ as $t\to\infty$ for some $\alpha>0$ one also speaks of \emph{polynomial stability} of the semigroup.  For a general overview of various notions of stability, and their relevance to damped waves, we refer the reader to the survey article~\cite{ChSeTo20}.

Our main result, Theorem~\ref{thm:poly_Datko}, is a non-uniform analogue of Theorem~\ref{thm:Datko} for {bounded} $C_0$-semigroups. It establishes polynomial as opposed to uniform exponential stability of the semigroup by requiring the integrability assumption in condition~(ii) to hold only for elements of the domain of a fractional power associated with the generator. In Theorem~\ref{thm:weak-Lp} we moreover establish a variant of our main result in which the integrability condition for semigroup orbits is replaced by an integrability condition for weak orbits. 
In Section~\ref{sec:Lyapunov} we give a sufficient condition for polynomial stability in terms of a non-uniform Lyapunov condition.
Finally, in Section~\ref{sec:obs}, we use Theorem~\ref{thm:poly_Datko} to prove polynomial stability of a semigroup under a certain non-uniform observability condition motivated by the study of (abstract) damped wave equations.

We use standard notation. In particular, if $X$ and $Y$ are (always complex) Banach spaces and $A\colon D(A)\subseteq X\to Y$ is a linear operator we denote  the domain of $ A$ by $D(A)$ and the resolvent set of $A$ by $\rho(A)$. The set of bounded linear operators from $X$ to $Y$ is denoted by $\B(X,Y)$, and we write $\B(X)$ for $\B(X,X)$. We moreover make use of standard asymptotic notation such as `big-O' and `little-o' notation, and for real-valued quantities $p$ and $q$, we use the notation $p\lesssim q$ if $p\leq K q$ for some constant $K > 0$ which is independent of all the parameters that are free to vary in the given situation. We let $\CC_+=\{\lambda\in\CC:\R\lambda>0\}$ denote the open right half-plane, and we write $\RR_+=[0,\infty)$ for the non-negative half-line.

\section{Non-uniform Datko--Pazy theorems}

Our main result is the following non-uniform version of Theorem~\ref{thm:Datko}.

\begin{thm}\label{thm:poly_Datko}
Let $\T$ be a bounded $C_0$-semigroup on a Banach space $X$, with generator $A$. Furthermore, let $1\le p<\infty$ and $\beta>0$, and suppose that $T(\cdot)x\in L^p(0,\infty;X)$ for all $x\in D((-A)^\beta)$. Then $i\RR\subseteq\rho(A)$ and 
\begin{equation}\label{eq:poly}
\|T(t)A\inv\|=O\big(t^{-1/(p\beta)}\big),\qquad t\to\infty.
\end{equation}
Moreover, if $X$ is a Hilbert space then $\|T(t)x\|=o(t^{-1/(p\beta)})$ as $t\to\infty$ for all $x\in D(A)$.
\end{thm}

\begin{proof}
The operator $(I-A)^{-\beta}$ maps $X$ bijectively onto $D((-A)^\beta)$, so the assumption that $T(\cdot)x\in L^p(0,\infty;X)$ for all $x\in D((-A)^\beta)$ is equivalent to requiring that $T(\cdot)(I-A)^{-\beta}x\in L^p(0,\infty;X)$ for all $x\in X$. Hence  $x\mapsto T(\cdot)(I-A)^{-\beta}x$ is a well-defined linear map from $X$ into $L^p(0,\infty;X)$, and it is straightforward to verify that its graph is closed. By the closed graph theorem there exists a constant $K>0$ such that 
$$\int_0^\infty\|T(t)(I-A)^{-\beta}x\|^p\,\dd t\le K^p\|x\|^p,\qquad x\in X.$$
Thus, letting $M=\sup_{t\ge0}\|T(t)\|$, we have
$$t\|T(t)(I-A)^{-\beta}x\|^p=\int_0^t\|T(t-s)T(s)(I-A)^{-\beta}x\|^p\,\dd s\le K^pM^p\|x\|^p$$
for all $x\in X$ and $t>0$, and hence $\|T(t)(I-A)^{-\beta}\|=O(t^{-1/p})$ as $t\to\infty$. It follows from the moment inequality as in~\cite[Prop.~3.1]{BEPS06} that $\|T(t)(I-A)^{-1}\|=O(t^{-1/(p\beta)})$ as $t\to\infty$, and hence $i\RR\subseteq\rho(A)$ by~\cite[Thm.~1.1]{BatDuy08}. Now~\eqref{eq:poly} follows since the operator $(I-A)A\inv$ is bounded, and the final claim follows from~\cite[Thm.~2.4]{BorTom10}.
\end{proof}

\begin{rem}\label{rem:converse}
\begin{enumerate}[(a)]
\item The conclusion of Theorem~\ref{thm:poly_Datko} implies in particular that the semigroup $\T$ is \emph{strongly stable} in the sense that $\|T(t)x\|\to0$ as $t\to\infty$ for all $x\in X$. This follows from a standard density argument using boundedness of the semigroup.
\item There is a straightforward partial converse of Theorem~\ref{thm:poly_Datko}. Indeed, if  $\|T(t)A\inv\|=O(t^{-\alpha})$ as $t\to\infty$ for some $\alpha>0$ and if $\beta>0$, then $T(\cdot)x\in L^p(0,\infty;X)$ for all $x\in D((-A)^\beta)$ provided that $p>(\alpha\beta)\inv$.  We leave open whether stronger converse statements hold.
\end{enumerate}
\end{rem}

Our next result is a non-uniform version of~\cite[Thm.~1.1]{Wei88}; see also \cite{Glu15}. Here we replace the integrability condition for orbits by a weak integrability condition but we nevertheless obtain the same conclusion in the Hilbert space setting; see Remark~\ref{rem:Banach} below for a slightly less sharp version of the result on general Banach spaces.

\begin{thm}
\label{thm:weak-Lp}
Let $\T$ be a bounded $C_0$-semigroup on a Hilbert space $X$, with generator $A$. Furthermore, let $1\le p <\infty$ and $\beta > 0$, and suppose that
\begin{equation}
\label{eq:weak_Lp}
\int_0^{\infty} |\langle T(t)x, y \rangle|^p \, \dd t <  \infty,\qquad x \in D((-A)^\beta),\ y \in X.
\end{equation}
Then $i\RR\subseteq\rho(A)$ and
\begin{equation}
\label{eq:sg_decay}
\| T(t) A\inv \| = O\big(t^{-1/(p\beta)}\big), \qquad t \to  \infty.
\end{equation}
Furthermore, $\| T(t) x\| = o\big(t^{-1/(p\beta)}\big)$ as $ t \to  \infty$ for all $x\in D(A)$.
\end{thm}

\begin{proof}
Let $\mu\in\rho(A)$ be fixed. Arguing as in the proof of Theorem~\ref{thm:poly_Datko}, this time applying the closed graph theorem twice, we see that  there exists a constant $K>0$ such that 
\begin{equation}
\label{eq:weak_bd}
\int_0^{\infty} |\langle T(t)(\mu-A)^{-\beta}x, y \rangle|^p \, \dd t\le K^p\|x\|^p\|y\|^p,\qquad x, y\in X.
\end{equation}
Let $\lambda\in\CC_+$. Using boundedness of the semigroup $\T$ and the integral representation of the resolvent, we have
\begin{equation}
\label{eq:weak_rep}| \langle (\lambda - A)^{-1}(\mu - A)^{-\beta}x, y\rangle | \leq \int_0^{\infty} e^{- (\R \lambda) t} | \langle T(t)(\mu - A)^{-\beta}  x, y\rangle | \, \dd t
\end{equation}
for all $x,y\in Y$. Suppose first that $p\in(1,\infty)$ and let $q\in(1,\infty)$ denote the Hölder conjugate of $p$. Applying Hölder's inequality in~\eqref{eq:weak_rep} and then using~\eqref{eq:weak_bd} we obtain the estimate
\begin{equation}
| \langle (\lambda - A)^{-1}(\mu - A)^{-\beta}x, y\rangle |\le \frac{K\|x\|\|y\|}{(q\R\lambda)^{1/q}},\qquad x,y\in X,
\end{equation}
which implies that $\|(\lambda - A)^{-1}(\mu - A)^{-\beta}\|\le K(\R\lambda)^{-1/q}$ for all $\lambda\in\CC_+$. Similarly, if $p=1$ we obtain $\|(\lambda - A)^{-1}(\mu - A)^{-\beta}\|\le K$ for all $\lambda\in\CC_+$. Now for $\lambda\in\rho(A)$ and $n\in\NN$ the resolvent identity gives
\begin{equation}
\label{eq:res_id}
(\lambda - A)^{-1} =  \sum_{k = 0}^{n - 1} (\mu-\lambda)^k (\mu -A)^{-(k+1)}  + ( \mu-\lambda)^n(\lambda - A)^{-1}(\mu - A)^{-n}.
\end{equation}
If  $n\in\NN$ satisfies $n>\beta$, then 
$$\|(\lambda - A)^{-1}(\mu-A)^{-n}\|\le\|(\mu-A)^{-(n-\beta)}\|\|(\lambda - A)^{-1}(\mu-A)^{-\beta}\|$$
for all $\lambda\in\rho(A)$. Fix $s\in\RR$. It follows from our previous estimates that for $\lambda=r+is$ with $r\in(0,1)$ we have 
$\|(\lambda - A)^{-1}\|\lesssim r^{-1/q}$ when $p>1$, and $\|(\lambda - A)^{-1}\|\lesssim 1$  when $p=1$. In particular,  we obtain $|\lambda-is|\|(\lambda-A)^{-1}\|<1$ for $\lambda=r+is$ with $r\in(0,1)$ sufficiently small, and hence $is\in\rho(A)$ by a standard Neumann series argument.

We now establish an estimate for the resolvent of $A$ along the imaginary axis. Modifying the choice of $\mu\in\rho(A)$ if necessary we obtain from the proof of~\cite[Lem.~3.2]{LatShv01} that there exists a constant $R\ge2$ such that 
\begin{equation}
\label{eq:LS_est}
\|(\lambda-A)^{-1}\|\lesssim|\lambda|^{\beta}\big(1+\|(\lambda-A)\inv(\mu-A)^{-\beta}\|\big)
\end{equation}
for all $\lambda\in\CC$ such that $0<\R\lambda<1$ and $|\lambda|\ge R$. If $p>1$ we thus obtain
$$\|(\lambda-A)^{-1}\|\lesssim|\lambda|^\beta\left(1+\frac{K}{(\R\lambda)^{1/q}}\right) $$
for $0<\R\lambda<1$ and $|\lambda|\ge R$, and hence $\|(\lambda-A)^{-1}\|\lesssim r^{-1/q} |s|^\beta $ for $\lambda=r+is$ with $r\in(0,1)$, $s\in\RR$ and $|\lambda|\ge R$. Setting $r=c|s|^{-p\beta}$ for sufficiently small $c>0$ we deduce that there exists $s_0>0$ such that $|\lambda-is|\|(\lambda-A)\inv\|\le1/2$ for all $s\in\RR$ with $|s|\ge s_0$, and hence by another Neumann series argument we obtain the resolvent estimate 
$\|(is-A)\inv\|\lesssim |s|^{p\beta}$ for such values of $s\in\RR$. On the other hand, if $p=1$ then~\eqref{eq:LS_est} combined with our earlier estimate for $\|(\lambda - A)^{-1}(\mu - A)^{-\beta}\|$ yields $\|(\lambda-A)\inv\|\lesssim|\lambda|^\beta$ for all $\lambda\in\CC$ 
such that $0<\R\lambda<1$ and $|\lambda|\ge R$. It follows by continuity that $\|(is-A)\inv\|\lesssim|s|^\beta$ for all $s\in\RR$ of sufficiently large absolute value. We have thus proved that $\|(is-A)\inv\|=O(|s|^{p\beta})$ as $|s|\to\infty$ for $1\le p<\infty$. Both of the remaining claims  now follow from~\cite[Thm.~2.4]{BorTom10}.
\end{proof}

\begin{rem}
\label{rem:Banach}
If we merely assume $X$ to be a Banach space, and replace~\eqref{eq:weak_Lp} by the condition
$$
\int_0^{\infty} |\langle T(t)x, x^* \rangle|^p \, \dd t <  \infty,\qquad x \in D((-A)^\beta),\ x^* \in X^*,$$ 
where $X^*$ denotes the dual space of $X$ and $\langle\cdot,\cdot\rangle$ is now the dual pairing between $X$ and $X^*$, then essentially the same argument as above shows that $i\RR\subseteq\rho(A)$ and $\|(is-A)\inv\|=O(|s|^{p\beta})$ as $|s|\to\infty$ for $1\le p<\infty$. It follows from~\cite[Thm.~1.5]{BatDuy08} that
$$\|T(t)A\inv\|=O\Bigg(\bigg(\frac{\log(t)}{t}\bigg)^{1/(p\beta)}\Bigg),\qquad t\to\infty.$$
Note that this estimate is worse by a logarithmic factor than the estimate in~\eqref{eq:sg_decay} obtained for the Hilbert space case.
\end{rem}

\section{A non-uniform Lyapunov-type condition}
\label{sec:Lyapunov}

Let $\T$ be a $C_0$-semigroup on a Hilbert space $X$, with generator $A$. The classical Lyapunov stability result states that the semigroup $\T$ is uniformly exponentially stable if and only if there exists a  non-negative self-adjoint operator $P\in\B(X)$ such that the \emph{Lyapunov equation}
\begin{equation}\label{eq:Lyapunov}
\langle PAx,y\rangle+\langle Px,Ay\rangle=-\langle x,y\rangle,\qquad x,y\in D(A),
\end{equation}
is satisfied; see for instance~\cite[Cor.~6.5.1]{CurZwa20}.
The following result shows that in the Hilbert space setting the integrability condition in Theorem~\ref{thm:poly_Datko} is equivalent to a non-uniform Lyapunov-type equation when $p=2$. In Corollary~\ref{cor:Lyapunov} below we combine this result with Theorem~\ref{thm:poly_Datko} to show that this non-uniform Lyapunov-type condition implies polynomial stability of the semigroup.
Here and elsewhere we regard the domain of an operator as a normed space with respect to the graph norm. We refer to the space of bounded \emph{conjugate-linear} functionals on a normed space as its \emph{antidual}.

\begin{prp}\label{prp:Lyapunov}
Let $\T$ be a bounded $C_0$-semigroup on a Hilbert space $X$, with generator $A$, and let $\beta>0$.  The following are equivalent:
\begin{enumerate}[(i)]
\item[\textup{(i)}]  $T(\cdot)x\in L^2(0,\infty;X)$ for all $x\in D((-A)^\beta);$
\item[\textup{(ii)}]\label{ii} There exists a bounded linear operator $P$ from $D((-A)^\beta)$ into its antidual such that $\langle Px,x\rangle_\beta\ge0$ for all $x\in D((-A)^\beta)$ and 
\begin{equation}\label{eq:P}
\langle PAx, y \rangle_\beta + \langle Px, Ay \rangle_\beta = - \langle x, y \rangle, \qquad x,y \in D((-A)^{\beta+1}),
\end{equation}
where $\langle \cdot,\cdot\rangle_\beta$ denotes the antiduality pairing for $D((-A)^\beta)$.
\end{enumerate}
Furthermore, the operator $P$ in item~\textup{(ii)} is unique and determined by 
\begin{equation}\label{eq:P_formula}
\langle Px, y \rangle_\beta = \int_0^{\infty} \langle T(t)x, T(t)y \rangle \, \dd t, \qquad x,y \in D((-A)^\beta).
\end{equation}
\end{prp}

\begin{proof}
Suppose that~(i) holds. Then  by Theorem~\ref{thm:poly_Datko} and Remark~\ref{rem:converse}(a) the semigroup $\T$ is strongly stable.  Furthermore, since $(-A)^\beta$ is an isomorphism from $D((-A)^\beta)$ onto $X$,  applying the Cauchy--Schwarz inequality  and applying the closed graph theorem as in the proof of Theorem~\ref{thm:poly_Datko} we see that there exists a constant $K>0$ such that
$$\int_0^\infty|\langle T(t)x,T(t)y\rangle|\,\dd t\le K\|(-A)^\beta x\|\|(-A)^\beta y\|,\qquad x,y\in D((-A)^\beta).$$
Thus if we define $P$ by the formula in~\eqref{eq:P_formula} then $P$ is indeed a bounded linear operator from $D((-A)^\beta)$ into its antidual, and it is clear that $\langle Px,x\rangle_\beta\ge0$ for all $x\in D((-A)^\beta)$. Let $x, y \in D((-A)^{\beta+1})$. Then
\begin{equation}
\label{eq:PAx}
\langle PAx, y \rangle_\beta = \int_0^{\infty} \langle T(t) Ax, T(t) y \rangle \, \dd t = \int_0^{\infty} \left \langle \frac{\dd}{\dd t}T(t)x, T(t)y \right \rangle\, \dd t,
\end{equation}
and similarly
\begin{equation}
\label{eq:Px}
\langle Px, Ay \rangle_\beta = \int_0^{\infty} \left \langle T(t)x, \frac{\dd}{\dd t} T(t)y \right \rangle \dd t.
\end{equation}
Using strong stability of $\T$ we deduce that
$$\langle PAx, y \rangle_\beta + \langle Px, Ay \rangle_\beta =\lim_{\tau\to\infty}\int_0^{\tau}\frac{\dd}{\dd t}  \langle T(t)x, T(t)y \rangle \, \dd t=-\langle x,y\rangle,$$
as required.

Now suppose, conversely, that~(ii) holds and let $V(x)=\langle Px,x\rangle_\beta$ for $x\in D((-A)^\beta)$. Then $V(x)\ge0$ for all $x\in D((-A)^\beta)$ and, since $P$ is assumed to be a bounded operator from  $D((-A)^\beta)$ into its antidual, there exists a constant $K>0$ such that
$$V(x)\le K\|(I-A)^\beta x\|^2,\qquad x\in D((-A)^\beta).$$
Now let $x\in D((-A)^{\beta+1})$. Then $$T(\cdot)x\in C^1\big(\RR_+,D((-A)^\beta))\big)\cap C\big(\RR_+,D((-A)^{\beta+1})\big),$$ 
so
$$\frac{\dd}{\dd t}V(T(t)x)=\langle PAT(t)x, T(t)x\rangle_\beta + \langle PT(t)x, AT(t)x\rangle_\beta = -\|T(t)x\|^2$$
for all $ t \geq 0$ by~\eqref{eq:P}. Thus
$$\int_0^\tau\|T(t)x\|^2\,\dd t=V(x)-V(T(\tau)x)\le V(x)\le K\|(I-A)^\beta x\|^2$$
for all $\tau>0$, and thus $T(\cdot)x\in L^2(0,\infty;X)$. Since $D((-A)^{\beta+1})$ is dense in $D((-A)^{\beta})$, an approximation argument yields $T(\cdot)x\in L^2(0,\infty;X)$ for all $x\in D((-A)^\beta)$, so~(i) holds.

It remains, finally, to show that the operator $P$ in~\textup{(ii)} is unique and determined by 
the formula in~\eqref{eq:P_formula}. Suppose therefore that $P$ is as described in~\textup{(ii)}   and let $x, y \in D((-A)^{\beta+1})$. Then using ~\eqref{eq:P} we have
\begin{equation}
\label{eq:diff-lyap-bis}
\frac{\dd}{\dd t} \langle PT(t)x, T(t)y\rangle_\beta = - \langle T(t)x, T(t)y \rangle,\qquad t\ge0,  
\end{equation} 
and therefore 
\begin{equation}
\label{eq:FTC}
\langle Px, y \rangle_\beta=\langle PT(t)x, T(t)y \rangle_\beta+\int_0^t\langle T(s)x, T(s)y \rangle\,\dd s,\qquad t\ge0.
\end{equation}
Since (i) holds we see, as in the proof of the implication from~(i) to~(ii), that the semigroup $\T$ is strongly stable, and hence $T(t)x,T(t)y\to0$ with respect to the graph norm on $D((-A)^\beta)$ as $t\to\infty$. It follows from continuity of $P$ that
$\langle PT(t)x, T(t)y \rangle_\beta \to 0$ as $t\to\infty$. Thus letting $t\to\infty$ in~\eqref{eq:FTC} shows that~\eqref{eq:P_formula} holds for all $x,y \in D((-A)^{\beta+1})$, and hence for all $x,y \in D((-A)^{\beta})$ by density.
\end{proof}

\begin{rem}
Note that boundedness of the semigroup $\T$ is not needed for the implication from~(ii) to~(i). 
\end{rem}

Proposition~\ref{prp:Lyapunov} is a natural non-uniform analogue of the classical Lyapunov stability result found for instance in~\cite[Lem.~4]{Dat70}, which can be combined with Theorem~\ref{thm:Datko} in order to recover the characterisation of exponential stability stated at the start of the section.
Here we obtain the following non-uniform analogue of this result.

\begin{cor}\label{cor:Lyapunov}
Let $\T$ be a bounded $C_0$-semigroup on a Hilbert space $X$, with generator $A$. Suppose that for some $\beta>0$ there exists a bounded linear operator $P$ from $D((-A)^\beta)$ into its antidual such that $\langle Px,x\rangle_\beta\ge0$ for all $x\in D((-A)^\beta)$ and 
\eqref{eq:P} holds. Then 
\begin{equation}\label{eq:poly}
\|T(t)x\|=o(t^{-\alpha}),\qquad t\to\infty
\end{equation}
 for all $x\in X$, where $\alpha=(2\beta)\inv$. 
 Conversely, if~\eqref{eq:poly} holds for some $\alpha>0$ and all $x\in X$, then for every $\beta>2\alpha\inv$ there exists a bounded linear operator $P$ from $D((-A)^\beta)$ into its antidual such that $\langle Px,x\rangle_\beta\ge0$ for all $x\in D((-A)^\beta)$ and 
\eqref{eq:P} holds. 
\end{cor}

\begin{proof}
The first implication follows immediately from Proposition~\ref{prp:Lyapunov} and Theorem~\ref{thm:poly_Datko}. Conversely, if \eqref{eq:poly} holds for some $\alpha>0$, then  $\|T(t)(I-A)\inv\|=O(t^{-\alpha})$ as $t\to\infty$ by the uniform boundedness principle, and hence $i\RR\subseteq\rho(A)$ and $\|T(t)A\inv\|=O(t^{-\alpha})$ as $t\to\infty$ as in the proof of Theorem~\ref{thm:poly_Datko}. It follows from Remark~\ref{rem:converse}(b) that $T(\cdot)x\in L^2(0,\infty;X)$ for all $x\in D((-A)^\beta)$ for $\beta>2\alpha\inv$. Now Proposition~\ref{prp:Lyapunov} gives the result.
\end{proof}

\section{An observability condition for polynomial stability}
\label{sec:obs}

In this final section we show how Theorem~\ref{thm:poly_Datko} can be used to deduce polynomial stability of a   semigroup from an observability-type condition. We begin by proving the following auxiliary result. 

\begin{lem}\label{lem}
Let $\T$ be a bounded $C_0$-semigroup on a Banach space $X$, with generator $A$, and let $Y$ be another Banach space. 
Furthermore, let $1\le p <\infty$ and let $C\in\B(D(A),Y)$, where $D(A)$ is endowed with the graph norm. Suppose there exists $p\in[1,\infty)$ such that for every $x\in D(A)$ the map $t\mapsto \|T(t)x\|^p$ is absolutely continuous on $(0,\infty)$, with 
\begin{equation}\label{eq:abs_cont}
\frac{\dd }{\dd t}\|T(t)x\|^p\le-\|CT(t)x\|^p
\end{equation}
for almost all $t\in(0,\infty)$, and suppose in addition that there exist constants $K,\beta,\tau>0$ such that
\begin{equation}\label{eq:obs}
\|(I-A)^{-\beta}x\|^p\le K\int_0^\tau\|CT(t)x\|^p\,\dd t,\qquad x\in D(A).
\end{equation}
Then $i\RR\subseteq\rho(A)$ and 
$$
\|T(t)A\inv\|=O\big(t^{-1/(p\beta)}\big),\qquad t\to\infty.
$$
Moreover, if $X$ is a Hilbert space then $\|T(t)x\|=o(t^{-1/(p\beta)})$ as $t\to\infty$ for all $x\in D(A)$.
\end{lem}

\begin{proof}
It follows from~\eqref{eq:abs_cont} that $\T$ is contractive and
\begin{equation}\label{eq:int_est}
\int_0^t\|CT(s)x\|^p\,\dd s\le \|x\|^p-\|T(t)x\|^p\le \|x\|^p
\end{equation}
for all $x\in D(A)$ and $t\ge0$. By the monotone convergence theorem  $ CT(\cdot)x\in L^p(0,\infty;X)$ for every $x\in D(A)$, and $\int_0^\infty\|CT(t)x\|^p\,\dd t\le \|x\|^p$. Moreover, for $k\in\ZZ_+$ and $t\in[k\tau,(k+1)\tau)$
  condition~\eqref{eq:obs} and boundedness of $\T$ imply that
$$\|T(t)(I-A)^{-\beta}x\|^p\le \|(I-A)^{-\beta}T(k\tau)x\|^p\le K\int_{k\tau}^{(k+1)\tau}\|CT(s)x\|^p\,\dd s$$
for all $x\in D(A)$.  Integrating over $[k\tau,(k+1)\tau)$ we obtain
$$\int_{k\tau}^{(k+1)\tau} \|T(t)(I-A)^{-\beta}x\|^p\,\dd t\le K \tau \int_{k\tau}^{(k+1)\tau}\|CT(t)x\|^p\,\dd t$$
for all $x\in D(A)$ and $k\in\ZZ_+$. Summing over all $k\in\ZZ_+$ gives
$$\int_0^\infty \|T(t)(I-A)^{-\beta}x\|^p\,\dd t\le K\tau \int_0^\infty\|CT(t)x\|^p\,\dd t\le K \tau\|x\|^p$$
for all $x\in D(A)$, and a  density argument yields $T(\cdot)(I-A)^{-\beta}x\in L^p(0,\infty;X)$ for all $x\in X$. Hence  $T(\cdot)x\in L^p(0,\infty;X)$ for all $x\in D((-A)^\beta)$, so the result follows from Theorem~\ref{thm:poly_Datko}.
\end{proof}

In the following result we suppose that $A$ is the generator of a $C_0$-semigroup $\T$ of contractions on a Hilbert space $X$ and that $B\in \B(U,X)$, where $U$ is another Hilbert space. Then the operator $A_B=A-BB^*$ with domain $D(A_B)=D(A)$ generates a $C_0$-semigroup $\TB$ of contractions on $X$. We may think of $\TB$ as a `damped' version of the semigroup $\T$. We refer the reader to~\cite{ChiPau23} for an in-depth study of the asymptotic behaviour of semigroups $\TB$ arising in this way (also under much milder assumptions on the operator $B$). The following result establishes polynomial decay of orbits $(T_B(t)x)_{t\ge0}$ for $x\in D(A_B)$ under a certain \emph{non-uniform observability condition}; see~\cite[Sec.~4]{ChiPau23}. This generalises~\cite[Thm.~4.4]{ChiPau23}, which included the additional assumption that $D(A)=D(A^*)$.

\begin{thm}\label{thm:obs}
Let $\T$ be a  $C_0$-semigroup of contractions on a Hilbert space $X$, with generator $A$, and let $U$ be another Hilbert space. Let $B\in\B(U,X)$ and let $\TB$ be the $C_0$-semigroup of contractions on $X$ generated by the operator $A_B=A-BB^*$  with domain $D(A_B)=D(A)$. Suppose there exist constants $K,\tau>0$ and $\beta\in(0,1]$ such that
\begin{equation}\label{eq:obs2}
\|(I-A)^{-\beta}x\|^2\le K\int_0^\tau\|B^*T(t)x\|^2\,\dd t,\qquad x\in X.
\end{equation}
Then $\|T_B(t)x\|=o(t^{1/(2\beta)})$ as $t\to\infty$ for all $x\in D(A_B)$.
\end{thm}

\begin{proof}
We begin by recalling from~\cite[Lem.~4.3]{ChiPau23} that 
\begin{equation}
\label{eq:var_const}
\int_0^\tau\|B^*T(t)x\|^2\,\dd t\lesssim \int_0^\tau\|B^*T_B(t)x\|^2\,\dd t,\qquad x\in X.
\end{equation}
 Since the semigroup $\T$ is assumed to be contractive, $D((I-A^*)^\beta)$ coincides with the complex interpolation space between $X$ and $D(A^*)=D(I-A^*)$ with parameter $\beta\in(0,1]$ by~\cite[Cor.~4.30]{Lun18}, and $D((I-A_B^*)^\beta)$ coincides with the complex interpolation space between $X$ and $D(A_B^*)$, in both cases with equivalent norms. Since $BB^*$ is a bounded linear operator on $X$, the spaces $D(A^*)$ and $D(A_B^*)$ are equal and their norms are equivalent. It follows that $D((I-A^*)^\beta)=D((I-A_B^*)^\beta)$ with equivalent norms, and hence
$\|(I-A^*)^\beta x\|\lesssim\|(I-A_B^*)^\beta x\|$ for all $x\in D((I-A^*)^\beta)$. Thus
$$\begin{aligned}
&\|(I-A_B)^{-\beta}x\|
=\sup_{y\in X\setminus\{0\}}\frac{|\langle x,(I-A_B^*)^{-\beta}y\rangle|}{\|y\|}\\&
\qquad =\sup_{y\in D((I-A_B^*)^\beta)\setminus\{0\}}\frac{|\langle x,y\rangle|}{\|(I-A_B^*)^\beta y\|}
\lesssim \sup_{y\in D((I-A^*)^\beta)\setminus\{0\}}\frac{|\langle x,y\rangle|}{\|(I-A^*)^\beta y\|}\\&
\qquad =\sup_{y\in X\setminus\{0\}}\frac{|\langle x,(I-A^*)^{-\beta}y\rangle|}{\|y\|}=\|(I-A)^{-\beta}x\|
\end{aligned}
$$
for all $x\in X$. Combining this estimate with~\eqref{eq:obs2} and~\eqref{eq:var_const} we obtain 
$$\|(I-A_B)^{-\beta}x\|^2\lesssim \int_0^\tau\|B^*T_B(t)x\|^2\,\dd t,\qquad x\in X.$$
Furthermore, for $x\in D(A_B)$ the map $t\mapsto\|T_B(t)x\|^2$ is differentiable on $(0,\infty)$ and dissipativity of $A$ yields
$$\frac{\dd }{\dd t}\|T_B(t)x\|^2\le-2\|B^*T_B(t)x\|^2,\qquad  t>0.$$
Hence the result follows from Lemma~\ref{lem} applied to the semigroup $\TB$ with $p=2$, $Y=U$ and $C=\sqrt{2}B^*$.
\end{proof}

We refer the reader to~\cite[Prop.~4.7]{ChiPau23} for an illustration of how Theorem~\ref{thm:obs} can be applied to a class of damped second-order systems which includes certain damped wave, beam and plate equations.

\bibliographystyle{plain}

\end{document}